\documentclass[12pt,thmsa, reqno]{amsart}
\usepackage{amsfonts}
\usepackage{amssymb, euscript, mathrsfs}

\newtheorem{theorem}{Theorem}[section]
\newtheorem{lemma}[theorem]{Lemma}

\newtheorem{remark}[theorem]{Remark}

\numberwithin{equation}{section}

\def\Ms{\mathscr{M m 1234}}

\def\Cs{\mathscr{C c 1234}}

\def\Cal{\mathcal}

\def\C{{\Cal C}}

\def\S{{\Cal S}}

\def\I{{\Cal I}}

\def\d{\partial}
\def\tr{{\hbox{\rm tr}}}

\def\Ma{\frM_{n,m}}

\def\Z{\mathcal{Z}}

\def\gnk{G_{n,k}}

\def\f0{f_0}
\def\Fc0{\varphi_0}

\def\I_k {I_{-}^{k/2}}
\def\I+k {I_{+}^{k/2}}

\def\vnk{V_{n,k}}
\def\vnm{V_{n,m}}

\def\cd{\stackrel{*}{\C}\!{}_{m, k}^\a}
\def\sd{\stackrel{*}{\S}\!{}_{m, k}^\a}
\def\cd0{\stackrel{*}{\C}\!{}_{m, k}^\a}
\def\sd0{\stackrel{*}{\S}\!{}_{m, k}^\a}
\def\fd{\stackrel{*}{F}\!{}_{\!m, k}}

\def\ncd0{\stackrel{*}{\Cs}\!{}_{m, k}^\a}

\def\bbr{{\Bbb R}}

\def\bbn{{\Bbb N}}

\def\bbc{{\Bbb C}}

\def\rank{{\hbox{\rm rank}}}

\def\tr{{\hbox{\rm tr}}}

\def\cos{{\hbox{\rm cos}}}
\def\det{{\hbox{\rm det}}}

\def\Pr{{\hbox{\rm Pr}}}

\def\gnk{G_{n,k}}
\def\gnm{G_{n,m}}

\def\part{\partial}
\def\intl{\int\limits}
\def\b{\beta}

\def\Gam{\Gamma}
\def\Om{\Omega}
\def\a{\alpha}
\def\om{\omega}

\def\Del{\Delta}
\def\del{\delta}
\def\vp{\varphi}

\def\gam{\gamma}
\def\Lam{\Lambda}
\def\sig{\sigma}
\def\lam{\lambda}

\def\t{\tau}

\font\frak=eufm10

\def\fr#1{\hbox{\frak #1}}

\def\frM{\fr{M}}

\def\cos{{\hbox{\rm cos}}}
\def\det{{\hbox{\rm det}}}

\def\gm{\Gamma_m}
\def\tr{{\hbox{\rm tr}}}
\def\part{\partial}
\def\intl{\int\limits}
\def\b{\beta}

\def\Gam{\Gamma}
\def\Om{\Omega}
\def\a{\alpha}

\def\sideremark#1{\ifvmode\leavevmode\fi\vadjust{\vbox to0pt{\vss
 \hbox to 0pt{\hskip\hsize\hskip1em
\vbox{\hsize2cm\tiny\raggedright\pretolerance10000
 \noindent #1\hfill}\hss}\vbox to8pt{\vfil}\vss}}}%

\newcommand{\be}{\begin{equation}}
\newcommand{\ee}{\end{equation}}
\newcommand{\bea}{\begin{eqnarray}}
\newcommand{\eea}{\end{eqnarray}}
\newcommand{\Bea}{\begin{eqnarray*}}
\newcommand{\Eea}{\end{eqnarray*}}

\begin{document}

\title[  ]
{Funk, Cosine, and Sine Transforms on Stiefel and Grassmann manifolds, II}

\author{  B. Rubin }
\address{Department of Mathematics, Louisiana State University, Baton Rouge,
Louisiana 70803, USA}
\email{borisr@math.lsu.edu}

\thanks{ The work was
 supported  by  NSF grants PFUND-137 (Louisiana Board of Regents) and 
 DMS-0556157.}





\begin{abstract}
We investigate analytic continuation of the matrix cosine and sine transforms introduced in Part I and  depending on a complex parameter $\a$. It is shown that the cosine transform corresponding to $\a=0$ is a constant multiple of the Funk-Radon  transform in integral geometry for a pair of Stiefel (or Grassmann) manifolds. The same case   for the sine transform gives the identity operator. These results and the relevant composition formula for the cosine transforms were established in Part I in the sense of distributions. Now we have them pointwise.
Some new problems are formulated.

 \end{abstract}

\maketitle

\section{Introduction}
\setcounter{equation}{0}

Let $\vnm$ and $\vnk$
 be  a pair of  Stiefel manifolds  of orthonormal frames in $\bbr^n$ of dimensions $m$ and $k$, respectively; $1\le m, k\le n-1$.
The following integral operators where introduced in \cite{Ru10}: \bea
\label{0mby}(\C^{\a}_{m, k} f)(u)&=&\int_{\vnm} \!\!\!f(v)\,
[\det \,(v'uu'v)]^{(\a-k)/2} \, d_*v,\\ \label{0mbyd}(\cd0
\vp)(v)&=&\int_{\vnk} \!\!\vp(u)\, [\det \,(v'uu'v)] ^{(\a-k)/2} \, d_*u,\\
\label{c0mby}(\S^{\a}_{m, k} f)(u)&=&\int_{\vnm} \!\!\!f(v)\, [\det \,(I_m
-v'uu'v)]^{(\a+k-n)/2} \, d_*v,\\
 \label{c0mbyd}(\sd0 \vp)(v)&=&\int_{\vnk} \!\!\vp(u)\,[\det \,(I_m
-v'uu'v)]^{(\a+k-n)/2}  \, d_*u.\eea
Here
$ u \!\in\! V_{n,k}$, $ v
\!\in\! \vnm$, $d_*u$ and $d_*v$ stand for the relevant probability measures, $(\cdot )'$ denotes the transpose of a matrix $(\cdot )$.

 We call $\C^{\a}_{m, k} f$ and $\S^{\a}_{m, k} f$ the {\it
cosine transform} and the {\it sine transform} of $f$, respectively.
 Integrals $\cd0 \vp$ and $\sd0 \vp$ are called the {\it dual cosine
transform} and the {\it dual sine transform}. The terminology stems from
the fact that, in the case $k=m=1$, when $u$ and $v$ are unit vectors, $$\det \,(v'uu'v)=(u\cdot v)^2=\cos^2 \om, \quad \det \,(I_m
-v'uu'v)=1-(u\cdot v)^2=\sin^2 \om,$$
 where $\om$ is the angle between $u$ and $v$. For integrable functions $f$ and $\vp$, integrals (\ref{0mby})-(\ref{c0mbyd}) converge absolutely almost everywhere
 if and only if $Re\,\a>m-1$ and represent analytic functions of $\a$ in this domain; see \cite[Theorem 4.2]{Ru10}.

The present article is a continuation of our previous work \cite{Ru10} devoted to the study of  operators (\ref{0mby})-(\ref{c0mbyd}). The impetus
 for this research was given by the celebrated Matheron's conjecture  in stochastic geometry   \cite[p. 189]{Mat} and a series of related publications in the area of harmonic analysis and integral geometry due to Goodey and Howard
\cite{GH1},  Alesker and Bernstein
\cite{AB},  Alesker \cite{A}, Ournycheva  and  Rubin \cite {OR4, OR3},  Zhang \cite{Zh2}; see \cite{Ru10} for the detailed exposition  of the history and motivation  of this research.

In the present paper we suggest an elementary approach to analytic continuation of the cosine transform, which enables us to study analytic properties of the sine transform and the duals of both transforms. We answer a series of questions stated in \cite{Ru10} and related to pointwise  equalities, the validity of which was proved in Part I only in the sense of distributions.

 In section 2 we recall basic definitions and auxiliary facts, which are used in the proofs.
 Main results are presented by Theorems
 \ref{lhgn},  \ref{lhgns},  \ref{lhdu} in sections 3,4 and 5, respectively. Section 6 contains some consequences, including pointwise inversion of the relevant Funk-Radon transform and a composition formula (Theorems \ref{cr72} and \ref {cr24n}). We conclude with brief discussion of the so-called $ \mbox{Cos}^\lam$- transform, which represents a particular case $k=m$ of (\ref{0mby}) and differs from  the latter by notation. Such transforms arise in group representations \cite{DM, OP, Pa}.

 {\bf Acknowledgements.} I am thankful to  Professors  Tomoyuki Kakehi, Gestur \`Olafsson,  Angela Pasquale, and  Genkai Zhang for useful discussions.

\section{Preliminaries}

More information about  facts presented below can be found in \cite{GR, OR04, Ru06, Ru10}.

\subsection{Notation and conventions} Given a square matrix $a$,  $|a|$ stands for the absolute value of
  $\det(a)$.  As usual,  $O(n)$   and $SO(n)$ stand for the orthogonal group and the
 special orthogonal group of $\bbr^{n}$, respectively, with the  normalized
 invariant measure of total mass 1. The abbreviation ``$a.c.$'' denotes
 analytic continuation.

 Let  $\frM_{n,m}\sim\bbr^{nm}$ be  the
space of real matrices $x=(x_{i,j})$ having $n$ rows and $m$
 columns;  $dx=\prod^{n}_{i=1}\prod^{m}_{j=1} dx_{i,j}$;
   $x'$ is  the transpose of  $x$, $|x|_m=\det
(x'x)^{1/2}$, $I_m$
   is the identity $m \times m$
  matrix, and $0$ stands for zero entries.

Let $\S_m \sim \bbr^{m(m+1)/2}$  be the space of $m \times m$ real
symmetric matrices $r=(r_{i,j})$; $dr=\prod_{i \le j} dr_{i,j}$. We
denote by  $\Omega$   the cone of positive definite matrices in $\S_m$.
The  Siegel gamma  function of $\Omega$
 is defined by
\be\label{2.4}
 \gm (\a)\!=\!\int_{\Omega} \!\exp(-\tr (r)) |r|^{\a -(m+1)/2}\, dr
 \!=\!\pi^{m(m-1)/4}\prod\limits_{j=0}^{m-1} \Gam (\a\!-\! j/2).  \ee
 This integral is absolutely convergent if and only if $Re
 \, \a >(m-1)/2$, and extends  meromorphically  with
 the  polar set  \be\label {09k}\{(m-1-j)/2: j=0,1,2,\ldots\};\ee see
\cite{Gi}, \cite{FK},
 \cite{T}.

For $n\geq m$, let $\vnm= \{v \in \frM_{n,m}: v'v=I_m \}$
 be  the Stiefel manifold  of orthonormal $m$-frames in $\bbr^n$. This is a homogeneous space
with respect to the action $\vnm \ni v\to \gam v $, $\gam\in O(n)$,
so that   $\vnm=O(n)/O(n-m)$.
We  fix a measure $dv$ on $\vnm$, which is left $O(n)$-invariant,
right $O(m)$-invariant, and
  normalized by \be\label{2.16} \sigma_{n,m}
 \equiv \int_{\vnm} dv = \frac {2^m \pi^{nm/2}} {\gm
 (n/2)}, \ee
\cite[p. 70]{Mu}. The notation $d_\ast v=\sig^{-1}_{n,m} dv$ is used for the corresponding probability measure.

We denote by $G_{n,m}$ the Grassmann manifold of $m$-dimensional linear subspaces $\xi$ of $\bbr^n$ equipped with the $O(n)$-invariant probability measure $d_*\xi$. Every right $O(m)$-invariant function $f (v)$  on $V_{n,m}$ can be identified with a function $\tilde f (\xi)$ by the formula
$\tilde f (\{v\})=f(v)$, $ \{v\}=v\bbr^m\in G_{n,m}$,  so that $\int_{G_{n,m}} \tilde f (\xi)d_*\xi=\int_{V_{n,m}} f (v)d_*v$. Another identification is also possible, namely, $\stackrel{*}{f} (\{v\}^\perp)=f(v)$, $ \{v\}^\perp\in G_{n,n-m}$.

\begin{lemma}\label{l2.3} {\rm (The polar decomposition).} Let $x \in \frM_{n,m}, \; n \ge m$. If  $\rank (x)=
m$, then
\be\label{pol} x=vr^{1/2}, \qquad v \in \vnm,   \qquad r=x'x \in\Omega,\ee
and $dx=2^{-m} |r|^{(n-m-1)/2} dr dv$.
\end{lemma}

For this statement see, e.g.,  \cite{Herz},  \cite{Mu},
\cite{FK}.

\subsection{Zeta integrals} Let $S(\Ma)$ be the Schwartz space of infinitely differentiable functions on $\Ma$ which are rapidly decreasing together with derivatives of all orders.  Suppose that $n\geq m\ge 2$ and
consider the  zeta   integral \be\label{zeta}
\Z(f,\a-n)=\intl_{\Ma} f(x) |x|^{\a-n}_m dx, \qquad f\in   S(\Ma).\ee

\begin{lemma}\label{lacz}
 The integral (\ref{zeta})
is absolutely convergent if $Re\, \a > m-1$ and  extends to $Re\, \a \leq m-1$
as a meromorphic function of $ \; \a$ with the only poles $ \;
m-1, m-2,\dots\;$. These poles and their orders are exactly the
same as of the gamma function $\gm(\a/2)$. The normalized  zeta
integral $ \Z(f,\a-n)/\gm(\a/2)$
 is an entire
function of $\a$.
\end{lemma}

This statement can be found in  \cite{Sh, Kh}; see also  \cite [Lemma 4.2]{Ru06}. The {\it Cayley-Laplace operator} $\Del$  on the space $\Ma$ of
matrices $x=(x_{i,j})$ is  defined by \be\label{K-L} \Del=\det(\d
'\d). \ee Here $\partial$ is an $n\times m$  matrix whose entries
are partial derivatives $\d/\d x_{i,j}$. More information about this operator, which is neither elliptic nor hyperbolic, can be found in \cite{Kh, Ru06}. The following identity of the Bernstein
type holds:
 \be\label{vaz}\Del ^\ell
|x|_m^{\a+2\ell-n}=B_\ell(\a)|x|_m^{\a-n},\ee  \be\label{bka}
B_{\ell,m,n}(\a)=\prod\limits_{i=0}^{m-1}\prod\limits_{j=0}^{\ell-1}(\a-i+2j)(\a-n+2+2j+i);\ee
see \cite[p. 565]{Ru06}. It allows us to represent  meromorphic continuation  of $\Z(f,\a-n)$ in the  form
\be\label {oao}
\Z(f,\a-n)=\frac{1}{B_\ell(\a)}\, \Z(\Del ^\ell f,\a+2\ell-n), \qquad Re \,\a > m-1-2\ell,\ee$$  \ell=1,2, \ldots, \, .$$

\begin{lemma}\label{tzk} \cite[Theorem 4.4]{Ru06} For $f  \in  \S(\Ma)$,
\be\label {1pqpx}
 \underset
{\a=0}{a.c.} \,\frac{\Z(f,\a-n)}{\gm(\a/2)}=\frac{\pi^{nm/2}}{\gm(n/2)}\,f(0). \ee
\end{lemma}

The general references related to zeta integrals are fundamental works by  
    Bopp and   Rubenthaler
\cite{BR},  Igusa \cite{Ig},  Shintani \cite{Shin}.
 
\subsection{The Funk Transform}
 The classical Funk transform  on the
unit sphere $S^{n-1}\subset \bbr^n$ is defined by
\be \label {7a3v}(Ff) (u)=\int_{\{v\in S^{n-1}: \, u\cdot v=0\}} f(v)\,d_u v,
\qquad u\in S^{n-1};\ee
see, e.g., \cite {GGG, He}. In \cite{Ru10} we introduced the following  generalization of (\ref{7a3v}), in which $u\in V_{n, k}$ and $v\in V_{n, m}$ are elements of the respective  Stiefel manifolds, $ 1\le k,m\le n-1$.
The {\it higher-rank Funk transform } sends a function $f$ on $V_{n, m}$ to a function $F_{m,k} f$ on $V_{n, k}$ by the formula
 \be \label {la3v}(F_{m,k} f) (u)=\int_{\{v\in V_{n, m}: \, u'v=0\}} f(v)\,d_u v,
\qquad u\!\in\! V_{n, k}.\ee
The corresponding dual transform
 \be \label {la3vd}(\fd \vp) (v)=\int_{\{u\in V_{n, k}: \, v'u=0\}} \vp(u)\,d_v u,
\qquad v\!\in\! V_{n, m},\ee
 acts in the opposite direction. The condition $u'v=0$ means that subspaces $u\bbr^k \in \gnk$ and $v\bbr^m \in \gnm$ are mutually orthogonal. Hence,    necessarily,  $$k+m\le n.$$ In the case $k=m$ we denote $F_m= F_{m,m}$.

 To give  our  transforms precise meaning,  we set $G=O(n)$,
 \bea
K_0&=&\left \{ \tau \in G : \tau = \left[\begin{array} {cc} \gam & 0
\\ 0 & I_{k}
\end{array} \right], \quad \gam  \in
O(n-k) \right \},\\
\label {klop7} \check K_0&=&\left \{ \rho \in G : \rho = \left[\begin{array} {cc}
\del & 0
\\ 0 & I_m \end{array} \right], \quad  \del
\in O(n-m) \right \}, \eea
\be\label {klop8} u_0= \left[\begin{array} {c}  0 \\  I_{k} \end{array} \right],
 \quad \check{u}_0\!= \! \left[\begin{array} {c}  I_k \\  0 \end{array}
\right]; \qquad  v_0\!=\! \left[\begin{array} {c}  0 \\  I_{m} \end{array} \right],
 \quad \check{v}_0\!=\! \left[\begin{array} {c}  I_m \\  0 \end{array}
\right ],\ee

\[ u_0,  \check{u}_0 \in  \vnk; \qquad  v_0,  \check{v}_0 \in \vnm.\]
Then (\ref{la3v}) and (\ref{la3vd}) can be explicitly written as \be\label{876a}
(F_{m, k}f) (u)=\int_{V_{n-k,m}} \!\!\!\!\!f\left(g_u
\left[\begin{array} {c} \om
\\0
\end{array} \right]\right)\,d_*\om\!=\!\int_{K_0} \!\!f(g_u \t
\check{v}_0)\,d\t,
 \ee
 \be \label{876ab}(\fd\vp) (v)=\int_{V_{n-m,k}} \!\!\!\!\!\vp\left(g_v
\left[\begin{array} {c} \theta
\\0
\end{array} \right]\right)\,d_*\theta\!=\!\int_{\check K_0} \!\!\vp(g_v \rho
\check{u}_0)\,d\rho,
 \ee
where  $ g_u$ and  $ g_v$ are orthogonal transformations satisfying $g_u u_0=u$
 and $g_v v_0=v$, respectively.

\begin{lemma} \cite [Lemma 3.2] {Ru10} Let $ 1\le k,m\le n-1$; $\,k+m\le n$. Then
\be\label{009a}\int_{V_{n, k}}(F_{m, k}f)(u)\, \vp(u)\,
d_*u=\int_{\vnm}f(v)\,(\fd \vp)(v)\, d_*v\ee provided that at
least one of these integrals is finite when $f$ and $\vp$ are
replaced by $|f|$ and $|\vp|$, respectively.
\end{lemma}

\subsection{Cosine and Sine transforms}

 When dealing with operators (\ref{0mby}) and (\ref{0mbyd}), we restrict our consideration to the case $m\le k$, because, if $m>k$, then $|v'uu'v|=0$ for all $v\in \vnm$ and all $u\in\vnk$. Similarly,  for (\ref{c0mby}) and (\ref{c0mbyd}), we assume  $m\le n-k$, because, if $m>n-k$, then
\be\label {kopz}|I_m -v'uu'v|=|I_m
-v'\Pr_{\{u\}}v|=|v' \Pr_{\{u\}^\perp } v|=|v' \tilde u \tilde u' v|=0
\ee
(here $\tilde u $ is an arbitrary $(n-k)$-frame  orthogonal to $ \{u\}=u\bbr^k$). The case $k=n$, when $ v'uu'v\equiv I_m$, is also not interesting.
Clearly, \be\label{robp}
(\S^{\a}_{m, k} f)(u)\!=\!(\C^{\a}_{m, n-k} f)(\tilde u)\!=\!\int_{\vnm} \!\!\!f(v)\, |v'\tilde u\tilde
 u'v|^{(\a-(n-k))/2} \, d_*v.\ee

The case $k=m$, when $\C^{\a}_{m, k}$ and $\S^{\a}_{m, k}$  coincide with
their duals, is of particular importance. In this case we denote \be
\label{nkmt}(M^{\a} f)(u)=\int_{\vnm} \!\!\!f(v)\,
|u'v|^{\a-m} \, d_*v,\quad 1\!\le\! m\! \le\! n\!-\!1,\qquad \quad\; \ee
 \be\label{0mbyr}(Q^\a f)(u)=\int_{\vnm}\!\!\!f(v)\,|I_m\!-\!v'
uu' v|^{(\a+m-n)/2}\, d_*v,\quad 2m\!\le \!n, \ee where $u \in V_{n,m}$.

\section{Analytic continuation of  the cosine  transform}

We use the following indirect procedure. Consider  the cosine transform
\be\label{hgh}
(\C^{\a}_{m, k} f)(u)=\int_{\vnm} \!\!\!f(v)\,
|v'uu'v|^{(\a-k)/2} \, d_*v, \qquad u\in V_{n,k},\ee
and set $u=\gam u_0$,  $u_0= \left[\begin{array} {c}  0 \\  I_{k} \end{array} \right]\in\! V_{n,k}$, $\gam \in SO(n)$. Changing variable $v=\gam w$, we obtain
$$
(\C^{\a}_{m, k} f)(\gam u_0)=\int_{\vnm} \!\!\!f_\gam (w)\,
|w'u_0u'_0w|^{(\a-k)/2} \, d_*w,
$$
where $f_\gam (w)=f(\gam w)$. Consider an auxiliary integral
$$
F_\a (\gam)=\int_{\frM_{n,m}} f_\gam (x (x'x)^{-1/2})\, |x'u_0u'_0x|^{(\a-k)/2} \psi (x'x)\,  \exp(-\tr (x'x))\,dx,
$$
where $\psi (r)$ is a nonnegative $C^\infty$ function on the cone $\Om$ having a compact support away from the boundary of $\Om$. Suppose that $f$ belongs to $ C^\infty (\vnm)$. Then the function
 \be\label {j7n9} \vp_\gam (x)\equiv f_\gam (x (x'x)^{-1/2})\, \psi (x'x)\,  \exp(-\tr (x'x))\ee
 belongs to $S  (\frM_{n,m})$ and is supported away from the surface $|x'x|=0$.
 Passing to polar coordinates $x=wr^{1/2}$, $w\in \vnm$, $r\in \Om$ (cf. (\ref{pol})), we obtain
 \be  F_\a (\gam)=\varkappa (\a)\, (\C^{\a}_{m, k} f)(\gam u_0), \ee
 \be\label {as34} \varkappa  (\a)=2^{-m} \sig_{n,m}\int_{\Om} |r|^{(\a-k+n)/2-d} \psi (r)\,  \exp(-\tr (r))\, dr,\ee
  $d=(m+1)/2$.
 Since $\varkappa (\a)$ and its reciprocal are entire functions of $\a$, which differ from zero for any complex $\a$, then  analyticity  of $\C^{\a}_{m, k} f$ is equivalent to that of $F_\a (\gam)$ and possible poles of both functions are at the same points and of the same orders.

The function  $F_\a (\gam)$ can be represented as a zeta integral, namely,
\bea  F_\a (\gam)&=&\int_{\frM_{n,m}}\vp_\gam (x) \, |x'u_0u'_0x|^{(\a-k)/2}\,dx\nonumber\\\label {pz67}&=&
\int_{\frM_{k,m}}|y|_m^{\a-k}\,\tilde\vp_\gam (y)\,dy= \Z(\tilde\vp_\gam, \a-k),\eea
where the function \be \label {45d} \tilde\vp_\gam (y)=\int_{\frM_{n-k,m}}
\vp_\gam  \left(\left[\begin{array} {c}  \eta \\  y \end{array} \right]\right)\, d \eta \ee
belongs to  $S  (\frM_{k,m})$;  cf. (\ref{zeta}).  Thus,
\be\label {lkjn}
(\C^{\a}_{m, k} f)(\gam u_0)=\varkappa (\a)^{-1}\,\Z(\tilde\vp_\gam, \a-k).
\ee

\begin{theorem}\label{lhgn}
Let  $1\le m\le k\le n-1$, $f\in C^\infty (\vnm)$.

\noindent {\rm (i)} If $Re\, \a > m-1$,  then the cosine transform $ (\C^{\a}_{m, k} f)(u)$ is represented by an
 absolutely convergent integral.

\noindent {\rm (ii)} For every $u \!\in\! V_{n,k}$, the function $ \a \to (\C^{\a}_{m, k} f)(u)$ extends to the domain
 $Re\, \a \leq m-1$
as a meromorphic function  with the only poles $ \;
m-1, m-2,\dots\;$. These poles and their orders are exactly the
same as of the gamma function $\gm(\a/2)$.

\noindent {\rm (iii)}  The meromorphic continuation  of $(\C^{\a}_{m, k} f)(u)$ can be represented in the  form
\be\label {oaoc}
(\C^{\a}_{m, k} f)(\gam u_0)=\frac{1}{\varkappa (\a)\, B_{\ell,m,k}(\a)}\, \Z(\Del ^\ell \tilde\vp_\gam,\a+2\ell-k),\ee $$
 Re \,\a > m-1-2\ell,\qquad  \ell=1,2, \ldots, \, ,$$
where $B_{\ell,m,k}(\a)$ is the Bernstein  polynomial
 \be\label{bkak}
B_{\ell,m,k}(\a)=\prod\limits_{i=0}^{m-1}\prod\limits_{j=0}^{\ell-1}(\a-i+2j)(\a-k+2+2j+i)\ee
 and $\tilde\vp_\gam$ is defined by (\ref{45d})-(\ref{j7n9}).

\noindent {\rm (iv)} The normalized
integral $ (\C^{\a}_{m, k} f)(u)/\gm(\a/2)$
 is an entire
function of $\a$ and belongs to $ C^\infty (\vnk)$ in the $u$-variable.

\noindent {\rm (v)} If, moreover, $ k+m\le n$, then
  \be\label{zn0x} \underset
{\a=0}{a.c.} \,\frac{(\C^{\a}_{m, k} f)(u)}{\gm(\a/2)}=c_{k,m}\,(F_{m, k}f) (u), \ee
where $$ c_{k,m}=\frac{\gm(n/2)}{\gm(k/2)\, \gm((n-k)/2)},$$
and $(F_{m, k}f) (u)$ is the Funk transform (\ref{876a}).
\end{theorem}
\begin{proof}
Statements {\rm (i)} - {\rm (iv)} are  immediate consequences of  Lemma \ref{lacz} and equality (\ref{oao}).  In particular, smoothness of the analytic continuation \be\label {m2z3} a.c. \,(\C^{\a}_{m, k} f)(u)/\gm(\a/2)\ee in the $u$-variable follows from (\ref{oaoc}), because $\Del ^\ell \tilde\vp_\gam$ is a smooth function of   $\gam \in SO(n)$. To prove {\rm (v)},  by (\ref{lkjn}) we have
\be\label {9t7}
\underset
{\a=0}{a.c.} \,\frac{(\C^{\a}_{m, k} f)(u)}{\gm(\a/2)}=\underset
{\a=0}{a.c.} \,\frac{1}{\varkappa (\a)}\, \frac{\Z(\tilde\vp_\gam, \a-k)}{\gm(\a/2)}.\ee
Hence, by (\ref{1pqpx}), (\ref{45d}), and (\ref{j7n9}),
\bea &&\underset
{\a=0}{a.c.} \,\frac{(\C^{\a}_{m, k} f)(u)}{\gm(\a/2)}=\frac{\pi^{km/2}}{\varkappa (0)\,\gm(k/2)}\, \tilde\vp_\gam (0)\nonumber\\
&&{}\nonumber\\&&=\!\frac{\pi^{km/2}}{\varkappa (0)\,\gm(k/2)}\intl_{\frM_{n-k,m}} \!\!\!\!f_\gam \left( \left[\begin{array} {c}  \eta \\  0 \end{array} \right] \,(\eta'\eta)^{-1/2}\right)\,
 \psi (\eta'\eta)\,  \exp(-\tr (\eta'\eta))\,d\eta. \nonumber\eea
Since $n-k\ge m$, we can use Lemma \ref {l2.3} to pass to polar coordinates and get an expression of the form $c\,I_1 \,I_2$, where
$$
I_1\!=\!\intl_{V_{n-k,m}}\!\!\! f_\gam \left( \left[\begin{array} {c}  \om \\  0 \end{array} \right]\right)\,d_*\om,\qquad I_2\!=\!\intl_{\Om} \! |r|^{(n-k)/2-d} \psi (r)\,  \exp(-\tr (r))\, dr,
$$
$$
c=\frac{\pi^{km/2}\, \sigma_{n-k,m}}{\sigma_{n,m}\,\gm(k/2) \,I_2}=\frac{\gm(n/2)}{\gm(k/2)\, \gm((n-k)/2)}\,\frac{1}{I_2}
$$
(use (\ref{2.16})). Integral $I_1$ is exactly the Funk transform (\ref{876a}). This proves (\ref{zn0x}).
\end{proof}

Some comments are in order.

\noindent {\bf 1.} Equality (\ref{zn0x})  was obtained in \cite{Ru10} in the weak sense, using the Fourier transform technique; cf.  formula (6.9) in that paper.  Our new proof  is much simpler and more informative.

\noindent {\bf 2.} The particular case $k=m$ in Theorem \ref{lhgn} characterizes analytic properties of the operator $M^{\a}f$ defined by  (\ref{nkmt}). For example,  (\ref{zn0x}) yields
  \be\label{zn0xq} \underset
{\a=0}{a.c.} \,\frac{(M^{\a}f)(u)}{\gm(\a/2)}=c_{m,m}\,(F_{m}f) (u), \ee
where $$ c_{m,m}=\frac{\gm(n/2)}{\gm(m/2)\, \gm((n-m)/2)}, \qquad 2m\le n,$$
and
 \be \label {la3vq}(F_{m} f) (u)=\int_{\{v\in V_{n, m}: \, u'v=0\}} f(v)\,d_u v,
\qquad u\!\in\! V_{n,m},\ee
is the Funk transform.

\noindent {\bf 3.} The particular case $f\equiv 1$, when
\be\label{mnvb} (\C^{\a}_{m, k} 1)(u)\!\equiv \!
 \intl_{V_{n, m}}\!\!\!|v'uu'v|^{(\a -k)/2} \,
dv\!=\!\frac{\Gam_m (n/2)}{\Gam_m (k/2)}\, \frac{\Gam_m (\a/{2})}{\Gam_m ((\a\! -\!k\!+\!n)/{2})} \ee
 (cf. (8.16) in \cite{Ru10}) illustrates  the role of the gamma function $\Gam_m (\a/{2})$ in the statement (ii) of Theorem \ref{lhgn}.

\noindent {\bf 4.}  It is known \cite {OR04, Ru06} that analytic continuation of the zeta integral (\ref{zeta}) at the so-called {\it Wallah set} $\a= 0,1,2, \ldots m-1$ is represented as a convolution with a positive measure, which can be explicitly evaluated.  We conjecture, that a similar result holds for the normalized cosine transform $(\C^{\a}_{m, k} f)(u)/\gm(\a/2)$ and the corresponding measure can be explicitly evaluated in the form, which does not contain the auxiliary function $\psi$ (so far, we succeeded in doing that only for $\a=0$).

\section{Analytic continuation of  the sine transform}\label {64g5}

The sine transform of a function $f(v)$ on $\vnm$ is a function $(\S^{\a}_{m, k} f)(u)$   defined by
\be\label{c0mbyq}(\S^{\a}_{m, k} f)(u)=\int_{\vnm} \!\!\!f(v)\, |I_m
-v'uu'v|^{(\a+k-n)/2} \, d_*v, \qquad u \!\in\! V_{n,k}.\ee
We assume $k+m \le n$, because otherwise $|I_m
-v'uu'v|=0$  for all $v\in \vnm$ and  $u\in\vnk$; see (\ref{kopz}).

\begin{theorem}\label{lhgns} Let  $1\le k,m \le n-1$, $k+m\le n$, $f\in C^\infty (\vnm)$.

\noindent {\rm (i)} If $Re\, \a > m-1$,  then the sine transform $(\S^{\a}_{m, k} f)(u)$  is represented by an
 absolutely convergent integral.

\noindent {\rm (ii)} For every $u \!\in\! V_{n,k}$, the function $ \a \to (\S^{\a}_{m, k} f)(u)$ extends to the domain
 $Re\, \a \leq m-1$
as a meromorphic function  with the only poles $ \;
m-1, m-2,\dots\;$. These poles and their orders are exactly the
same as of the gamma function $\gm(\a/2)$.

\noindent {\rm (iii)}  The normalized
integral $ (\S^{\a}_{m, k} f)(u)/\gm(\a/2)$
 is an entire
function of $\a$ and belongs to $ C^\infty (\vnk)$ in the $u$-variable.

\noindent {\rm (iv)} In the case $k=m$, $2m\le n$, when  $\S^{\a}_{m, k} f\equiv Q^{\a} f$ is the integral (\ref{0mbyr}),  we have
\be\label{zn0q} \underset
{\a=0}{a.c.} \,\frac{(Q^{\a} f)(u)}{\gm(\a/2)}=\frac{\gm(n/2)}{\gm(m/2)\, \gm((n-m)/2)}\,f (u) \ee
provided that $f$ is right $O(m)$-invariant.
\end{theorem}
\begin{proof} We have
\[|I_m -v'uu'v|=|I_m
-v'\Pr_{\{u\}}v|=|v' \Pr_{\{u\}^\perp } v|=|v' \tilde u \tilde u' v|=0,
\]
where $\tilde u $ is an arbitrary $(n-k)$-frame  orthogonal to $ \{u\}=u\bbr^k$. Hence,
\be\label{c0mbt}(\S^{\a}_{m, k} f)(u)=\int_{\vnm} \!\!\!f(v)\, |v' \tilde u \tilde u' v|^{(\a+k-n)/2} \, d_*v\!=\!(\C^{\a}_{m, n-k} f)(\tilde u).
\ee
It remains to apply Theorem \ref{lhgn}.  To prove (\ref{zn0q}), we make use of (\ref{robp}) with $k=m$. Then, as above,
$$
\underset{\a=0}{a.c.} \,\frac{(Q^{\a} f)(u))}{\gm(\a/2)}=\underset
{\a=0}{a.c.} \,\frac{(\C^{\a}_{m, n-m} f)(\tilde u)}{\gm(\a/2)}=\tilde c\,\tilde I_1 \,\tilde I_2,$$ where
$$
\tilde I_1=\intl_{V_{m,m}}\!\!\! f_\gam \left( \left[\begin{array} {c}  \om \\  0 \end{array} \right]\right)\,d_*\om=\intl_{O(m)}\!\!\! f \left(  \gam \left[\begin{array} {c} I_m \\  0 \end{array} \right] \,\om\right)\,d\om=
f \left( \gam\left[\begin{array} {c}  I_m\\  0 \end{array} \right]\right),
$$
$$
I_2\!=\!\intl_{\Om} \! |r|^{m/2-d} \psi (r)\,  \exp(-\tr (r))\, dr, \quad \tilde c=\frac{\gm(n/2)}{\gm(m/2)\, \gm((n-m)/2)}\,\frac{1}{\tilde I_2}.$$
Here $\gam$ is a rotation that takes $\left[\begin{array} {c}  0\\  I_{n-m} \end{array} \right]$ to $\tilde u \in \{u\}^\perp$. Hence, $\gam$  takes $\left[\begin{array} {c}  I_m\\  0 \end{array} \right]$ to a certain frame, say,   $u_1$, that spans the same subspace as $u$.  Since $f$ is right $O(m)$-invariant, then $\tilde I_1\!=\!f(u)$, and we are done.
\end{proof}

\section{Analytic continuation of  the dual cosine  transform}
The dual cosine  transform of a function $\vp$ on $\vnk$ is    defined by
\be\label{ydcs}(\cd0
\vp)(v)\!=\!\int_{\vnk} \!\!\vp(u)\, |v'uu'v|^{(\a-k)/2} \, d_*u, \qquad v \!\in \!\vnm, \quad m\!\le \!k.\ee
We will be dealing with smooth functions $\vp\in C^\infty (\vnk)$. Moreover, since $\vp(u)$  and $\vp(u\gam)$ $\forall \gam \in O(k)$ have the same dual cosine  transform,  we  can assume that $\vp$ is right $O(k)$-invariant.
We denote by $\tilde u \in V_{n, n-k}$ and $\tilde v\in V_{n, n-m} $  arbitrary frames  orthogonal to subspaces $ \{u\}$ and $ \{v\}$, respectively. Then
\bea
|v'uu'v|&=&|I_m-v'\tilde u \tilde u'v|=|I_{n-k}-\tilde u'vv'\tilde u |=|I_{n-k}-\tilde u'\Pr_{\{v\}} \tilde u |\nonumber\\
&=&|\tilde u'\Pr_{\{v\}^\perp }  \tilde u |=|\tilde u'\tilde v \tilde v'  \tilde u |. \nonumber\eea
Setting $\vp_1(\tilde u) = \vp(u)$, we obtain
\bea\label {k580}(\cd0
\vp)(v)&=&\int_{ V_{n, n-k}} \vp_1(\tilde u) \, |\tilde u'\tilde v \tilde v'  \tilde u |^{(\a-k)/2} \, d_*\tilde u\\&=&
(\C^{\a+n-k-m}_{n-k, n-m}  \vp_1)(\tilde v ), \qquad \tilde v \in V_{n, n-m}.\nonumber\eea
Thus,  analytic properties of $\cd0
\vp$ can be derived from Theorem \ref{lhgn}.

\begin{theorem}\label{lhdu}
Let  $1\le m\le k\le n-1$ and let  $\vp$ be a right $O(k)$-invariant function in $C^\infty (\vnk)$.

\noindent {\rm (i)} If $Re\, \a > m-1$,  then  dual cosine transform $(\cd0 \vp)(v)$ is represented by an
 absolutely convergent integral.

\noindent {\rm (ii)} For every $v \!\in\! V_{n,m}$, the function $ \a \to (\cd0 \vp)(v))$ extends to the domain
 $Re\, \a \leq m-1$
as a meromorphic function  with the only poles $ \;
m-1, m-2,\dots\;$. These poles and their orders are exactly the
same as of the gamma function $\Gam_{n-k}((\a+n-k-m)/2)$.

\noindent {\rm (iii)}  The normalized
integral $ (\cd0 \vp)(v)/\Gam_{n-k}((\a+n-k-m)/2)$
 is an entire
function of $\a$ and belongs to $ C^\infty (\vnm)$ in the $v$-variable.

\noindent {\rm (iv)} If $ k+m\le n$, then $ (\cd0 \vp)(v)/\gm(\a/2)$ extends
as a meromorphic function  with the only possible poles $ \;
-1, -2,\dots\;$. Moreover, for every $v\in\vnm$,
  \be\label{zn0xdr} \underset
{\a=0}{a.c.} \,\frac{ (\cd0 \vp)(v)}{\gm(\a/2)}=c_{k,m}\,(\fd \vp) (v), \ee
where
$$ c_{k,m}=\frac{\gm(n/2)}{\gm(k/2)\, \gm((n-k)/2)}$$
and $(\fd \vp) (v)$ is the dual Funk transform (\ref{la3vd}).

\noindent {\rm (v)} If $ k+m> n$, then for every $v\in\vnm$,
  \be\label{zn0xd} \underset
{\a=k+m-n}{a.c.} \,\frac{ (\cd0 \vp)(v)}{\Gam_{n-k}((\a+n-k-m)/2)}=
 \tilde c_{k,m}\,(F_{n-k, n-m}  \vp_1)(\tilde v ), \ee
where
$$ \tilde c_{k,m}=\frac{\Gam_{n-k}(n/2)}{\Gam_{n-k}(m/2)\,\Gam_{n-k}((n-m)/2)}$$
and $(F_{n-k, n-m}  \vp_1)(\tilde v )$ is the relevant Funk transform; cf. (\ref{876a}).
\end{theorem}
\begin{proof} Owing to (\ref{k580}), statements {\rm (i)}-{\rm (iii)} and {\rm (v)} are immediate consequences of the respective statements in Theorem \ref{lhgn}.
To prove {\rm (iv)}, we observe that by (\ref{2.4}),
\be \Gam_{n-k}((\a+n-k-m)/2)= c (\a)\, \gm(\a/2), \ee
where $c (\a)= \pi^{(n-k-m)m/2}\Gam_{n-k-m}((\a+n-k-m)/2)$ is a meromorphic function with the polar set $ \{
-1, -2,\dots \}$. Denote
$$
A_\a(v)\equiv\frac{ (\cd0 \vp)(v)}{\gm(\a/2)}=\frac{c (\a)\, (\cd0 \vp)(v)}{\Gam_{n-k}((\a+n-k-m)/2)}, \quad Re\, \a>m-1.$$
By   {\rm (iii)},
 this function extends   analytically to $Re\, \a>-1$ and the analytic continuation  belongs to  $ C^\infty (\vnm)$. Now for any test function $\om  \in C^\infty (\vnm)$, owing to (\ref{zn0x}) and (\ref{009a}), we have
\bea
(\underset {\a=0}{a.c.} \,A_\a, \om)&=& \underset {\a=0}{a.c.} \,\left (\vp, \frac{\C^{\a}_{m, k} \om}{\gm(\a/2)}\right )=\left (\vp,  \underset {\a=0}{a.c.} \,\frac{\C^{\a}_{m, k} \om}{\gm(\a/2)}\right )\nonumber\\
&=&c_{k,m}\,(\vp, F_{m, k} \om)=c_{k,m}\,(\fd \vp,  \om),\nonumber\eea
and (\ref{zn0xdr}) follows.
\end{proof}

\section{Some consequences}

\subsection{Inversion of the Funk transform}

\begin{theorem} \label{cr72} Let $\vp=F_{m,k} f$, where $f$ is a $C^\infty$ right
$O(m)$-invariant function on $\vnm$,  $1\le m\le k\le n-m$. Then, for every $v\in \vnm$,
\be \label {forq1y}  \underset
{\a=k+m-n}{a.c.} \frac{(\cd0 \vp)(v)}{\Gam_m(\a/2)}=c\, f(v), \qquad c \!=\!\frac{\gm(n/2)}{\Gam_m (k/2)\, \gm(m/2)}.\ee
\end{theorem}
\begin{proof} By \cite [Theorem 4.3]{Ru10}, if $Re\, \a >m-1$, then
\be\label{gty} \frac{(\cd0 \vp)(v)}{\Gam_m(\a/2)}=\frac{\Gam_m((n-m)/2)}{\Gam_m(k/2)} \, \frac{(Q^{\a+n-k-m} f)(v)}{\Gam_m((\a+n-k-m)/2)},\ee
 where $Q^{\a+n-k-m} f$ is the sine transform from Section \ref {64g5}.
By Theorem \ref{lhgns}, the right-hand side of (\ref{gty}) extends as an entire function of $\a$. Now, taking analytic continuation  and using (\ref{zn0q}), we obtain the result.
\end{proof}

\begin{remark}{\rm It is interesting to note that in general, the function 
$ \a \to (\cd0 \vp)(v)/\gm(\a/2)$ admits  poles $ \;
-1, -2,\dots\;$; see Theorem \ref{lhdu}  (iv). However, as (\ref{gty}) reveals, these poles  disappear if $\cd0$ acts on the image of the Funk transform $F_{m,k}$.}
\end{remark}

\begin{remark}{\rm Theorem \ref{cr72} can be reformulated in the language of Grassmannians. The reader can  easily do this, using connection formulas from \cite[Section 3.2]{Ru10}. A series of inversion results for the  Funk-Radon transform on Grassmannians can be found in  the fundamental works by Gelfand and his collaborators \cite{GGR, GGS}, Grinberg \cite {Gri},  Kakehi \cite {Ka}; see also a pioneering paper by Petrov  \cite{P67},  Grinberg and  Rubin \cite {GR}, and  Zhang
\cite {Zh1}.  Diverse problems of integral geometry related to Grassmann manifolds were  studied in \cite {GGo, Go, GK1, GK2, He, Ru04, Shi}.
 Our  Theorem \ref{cr72} has a completely different flavor,  agrees with the known case $m=1$ (cf. formula (1.13) in \cite{Ru02}) and sheds new light to this circle of  problems.}
\end{remark}

\subsection{A composition formula; the case $k=m$}

In this particular case the cosine transform $(M^{\a}f)(u)=(\C^{\a}_{m, m} f)(u)$, $u \!\in\! V_{n,m}$, has a number of important features. We normalize it by setting
\bea (\Ms^{\a}f)(u)&=&\del_{n,m} (\a)\, (M^{\a}f)(u)\nonumber\\
\label{nmn}&=& \del_{n,m} (\a)\!\int_{\vnm} \!\!\!f(v)\, |u'v|^{\a -m} \,
d_*v,  \eea
\be \del_{n,m}
(\a)=\frac{\Gam_m(m/2)}{\Gam_m(n/2)}\,\frac{\Gam_m((m-\a)/2)}{\Gam_m(\a/2)}, \qquad \a\notin \bbn=\{1,2, \ldots\,\}.\nonumber\ee
The excluded values of $\a$ form the polar set of $\Gam_m((m-\a)/2)$. Integral (\ref{nmn}) converges absolutely for any integrable function $f$ when $Re \, \a> m-1$. If $f\in C^\infty (\vnm)$, then, by  Theorem \ref{lhgn}, analytic continuation of a function $\a \to (\Ms^{\a}f)(u)$ is well-defined for all complex $\a\neq 1,2, \ldots\,$ and belongs to $C^\infty (\vnm)$. We denote
\be
(\Ms_{a.c.}^{\a}f)(u)= \underset
{\a\notin \bbn}{a.c.} \,(\Ms^{\a}f)(u).\ee

\begin{theorem} \label{cr24n} Let $f\in C^\infty (\vnm)$ be a  right
$O(m)$-invariant function on $\vnm$, $2m \le n$. If $\a, 2m -\a -n \notin \bbn$, then for every $u\in \vnm$,
\be \label {for1n} (\Ms_{a.c.}^{2m -\a -n} \Ms_{a.c.}^{\a}f)(u)=f(u).\ee
\end{theorem}
\begin{proof} In \cite[Theorem 6.4]{Ru10} we proved that  if  $Re \, \a> m-1$, $ \a\neq m, m+1, m+2, \ldots$, then 
  \be \label {for1}  \underset
{\b=2m -\a -n}{a.c.} \,(\Ms^{\b} \Ms^{\a}f, \om)=(f,
\om), \ee
 for any test function $\om \in C^\infty (\vnm)$. Since $\Ms^{\a}f \in C^\infty (\vnm)$, then, by Theorem \ref{lhgn}, analytic continuation of a function $\b \to (\Ms^{\b} \Ms^{\a}f)(u)$ is well-defined for all $\b\neq 1,2, \ldots\,$ and represents a $C^\infty$  function on $\vnm$. Hence, $$(\Ms_{a.c.}^{\b} \Ms^{\a}f, \om)=({a.c.}\,\Ms^{\b} \Ms^{\a}f, \om)={a.c.}(\Ms^{\b} \Ms^{\a}f, \om),$$ and
 the aforementioned result from
\cite{Ru10} yields $(\Ms_{a.c.}^{2m -\a -n} \Ms^{\a}f, \om)=(f,\om)$. Since,  by Theorem \ref{lhgn}, the function  $(\Ms_{a.c.}^{\b} \Ms^{\a}f)(u)$ is smooth, then we have a pointwise equality
\be\label {jjdl}
(\Ms_{a.c.}^{2m -\a -n} \Ms^{\a}f)(u)=f(u), \ee
provided that
$$Re \, \a> m-1, \qquad \a\neq m, m+1, m+2, \ldots\, .$$
To extend this result to all complex $\a$ such that $\a, 2m -\a -n \notin \bbn$, we observe that $(\Ms_{a.c.}^{2m -\a -n} \Ms_{a.c.}^{\a}f)(u)$ is separately analytic  in $\a$ and $\b$ in the domain $\Lam =\{(\a,\b) \in \bbc^2: \a\notin \bbn, \; \b\notin \bbn \}$. Hence, by   the fundamental Hartogs theorem \cite {Sha}, this function is analytic in $z=(\a,\b)\in \Lam$.  Now   (\ref{for1n}) follows from (\ref{jjdl}) by the uniqueness of analytic continuation.
\end{proof}

\subsection{The $ \mbox{Cos}^\lam$-transforms}

 Transformation (\ref{nkmt}) can be found in the literature under different names (or without naming) and with different notation. For instance, let
\be\label {lvbj}
(\mbox{Cos}^\lam f)(u)=\int_{\vnm} \!\!\!f(v)\,
|u'v|^{\lam -\rho} \, d_*v, \ee$$
Re \, \lam>\rho -1,\qquad  \rho=n/2,\qquad  u\in \vnm;
$$
cf. \cite{OP}, where a similar operator is presented in a slightly different form. Following \cite{OP}, we call (\ref{lvbj})  the $ \mbox{Cos}^\lam$-transform of $f$. The connection between (\ref{lvbj}) and (\ref{nkmt}) is
\be
\mbox{Cos}^\lam f=M^{\lam +m-\rho} f.
\ee
In terms of  (\ref{lvbj}) formula (\ref{zn0xq}) becomes
\be\label{zn023} \underset 
{\lam =\rho -m}{a.c.} \,\frac{\mbox{Cos}^\lam f}{\gm((\lam +m-\rho)/2)}=c_{m,m}\,F_{m}f. \ee
 The
 composition formula (\ref{for1n}) transforms to
\be
\mbox{\it Cos}^{-\lam} \mbox{\it Cos}^\lam f =f, \qquad \pm\lam +m-\rho\neq 1,2, \ldots\, ,
\ee
where $\mbox{\it Cos}^\lam$ denotes analytic continuation of the normalized integral
\be
(\mbox{\it Cos}^\lam f)(u)= \tilde \del_{n,m}
(\lam)\int_{\vnm} \!\!\!f(v)\,
|u'v|^{\lam -\rho} \, d_*v.
\ee
\be  \tilde \del_{n,m}
(\lam)=\frac{\Gam_m(m/2)}{\Gam_m(\rho)}\,\frac{\Gam_m((\rho-\lam)/2)}{\Gam_m((\lam +m-\rho)/2)}.\nonumber\ee

\subsubsection{The case $m=1$} This case deserves special mentioning. The corresponding  cosine transforms on the unit sphere $S^{n-1}$ in $\bbr^n$ are well-known in analysis.
We recall some results; see a survey  article \cite{{Ru03}} and references therein. In the ``$\lam$-notation'' the normalized ``rank-one'' $ \mbox{Cos}^\lam$-transform is defined as analytic continuation of the  integral
\be
(\mbox{\it Cos}^\lam f)(u)= \frac{\pi^{1/2}\, \Gam ((\rho-\lam)/2)}{\Gam (\rho)\,\Gam ((\lam -\rho+1)/2)}
\int_{S^{n-1}} \!\!\!f(v)\,
|u\cdot v|^{\lam -\rho} \, d_*v,
\ee
 where $\rho=n/2$ and $d_*v$ is the probability measure on  $S^{n-1}$.  Let
$\{ Y_{j, \nu} (v) \}$ be an orthonormal basis of spherical
harmonics on $S^{n-1}$.  Here $j = 0, 1, 2, \dots $, and
$\nu = 1, 2, \dots, d_n (j)$, where $d_n (j)$ is the dimension of the
subspace of spherical harmonics of degree $j$. If    \be f = \sum\limits_{j, k} f_{j, k} Y_{j, k}, \qquad  f_{j, k}=
\int_{S^{n-1}} f(v)\, Y_{j, k}(v) \,dv,\ee (the Fourier-Laplace expansion
of $f$), then
\be \mbox{\it Cos}^\lam f=\sum\limits_{j, k} c_{j,\lam} f_{j, k} Y_{j, k},
\ee
where
\be  c_{j,\lam} =\left\{
\begin{array}{cl} (-1)^{j/2} \,\displaystyle{{\Gamma
((j+\rho -\lam)/2)\over \Gamma ((j+\rho +\lam)/2)} }
&  \mbox{\rm if $j$ is even}, \\
0 &  \mbox{\rm if $j$ is odd},
\end{array}
\right. \ee

Let $L_{p, \text{e}}\equiv L_{p, \text{e}}(S^{n-1})$ be the Lebesgue space of even $p$-integrable functions on $S^{n-1}$.
The asymptotic of $ c_{j,\lam}$ as $j \to \infty$ implies the following embeddings of the range $\mbox{\it Cos}^\lam ( L_{p, \text{e}})$, $1<p<\infty$, into the relevant  Sobolev spaces $L^\gamma_{p, \text{e}}$.

\begin{theorem} {\rm (cf. Corollary 3.3 in \cite{Ru03})} \ Let
$$
 \gamma_{\pm} =Re \, \lam  \pm \Big |
\frac1{p} - \frac12 \Big | (n -1),$$
\[
\lam \notin \{ \rho, \rho +2, \rho+4, \dots \} \cup \{ - \rho -1, - \rho - 3, -\rho - 5,
\dots \}.
\]
The following proper embeddings hold:
\be L^{\gamma_{-}}_{p, \text{e}} \subset \mbox{\it Cos}^\lam ( L_{p, \text{e}}) \subset L^{\gamma_{+}}_{p, \text{e}}. \ee
If $p=2$, then
\[
\mbox{\it Cos}^\lam  (L_{2,\text{e}})=L^{Re \, \lam} _{2, \text{e}}.
\]
\end{theorem}

It is a challenging problem to possibly extend these results to the higher-rank case $m>1$.

\bibliographystyle{amsplain}

\begin{thebibliography}{10}



\bibitem  {A} S. Alesker, \textit{The $\a$-cosine transform and
intertwining integrals}, Preprint, 2003.

\bibitem  {AB} S. Alesker and J. Bernstein, \textit{Range
characterization of the cosine transform on higher Grassmannians},
Advances in Math. \textbf{184} (2004), 367--379.

\bibitem {BR} N. Bopp and  H. Rubenthaler, Local Zeta Functions Attached to the Minimal Spherical Series for a Class of Symmetric Spaces, Memoirs of the Amer. Math. Soc., vol. 174, 2005.


\bibitem  {DM} G. Van Dijk and V.F. Molchanov,  \textit{Tensor products of maximal degenerate series representations of the group $SL(n;R)$},  J. Math. Pures Appl. (9) \textbf{78} (1999),  99--119.


\bibitem  {FK} J. Faraut and A. Kor\'anyi,  Analysis on symmetric cones, Clarendon Press, Oxford, 1994.

\bibitem  {GGo} J. Gasqui and H. Goldschmidt, 	
Radon transforms and the rigidity of the Grassmannians, Princeton Univ. Press, 2004.


\bibitem  {GGG}
 I.M. Gel'fand,   S.G. Gindikin and  M.I. Graev,  Selected topics
in integral geometry, Translations of Mathematical Monographs, AMS,
Providence, Rhode Island, 2003.

\bibitem  {GGR}  I. M. Gelfand,  M.I. Graev,  and  R. Ro\c{s}u, \textit{The problem of integral geometry
and intertwining operators for a pair of real Grassmannian
manifolds}, J. Operator Theory \textbf{12} (1984), 339--383.


\bibitem {GGS}I.M. Gelfand, M.I. Graev, and Z.Ja. \v{S}apiro,
\textit{A problem of integral geometry connected with a pair of
Grassmann manifolds}, Dokl. Akad. Nauk SSSR \textbf{193}, No. 2,
(1970), 892--896.



\bibitem {Gi} S.G. Gindikin, \textit{Analysis on homogeneous
domains}, Russian Math. Surveys \textbf{19}  (1964), No. 4, 1--89.

\bibitem  {Go} F. Gonzalez,   \textit{ Radon transform
on  Grassmann manifolds}, Journal of Func. Anal. \textbf{71}
(1987), 339--362.


\bibitem  {GK1} F. Gonzalez  and T. Kakehi,  \textit{Pfaffian systems and Radon
transforms on affine Grassmann manifolds},  Math. Ann.   \textbf{326}  (2003),  237--273.

\bibitem  {GK2} \bysame, Moment conditions and support theorems for Radon transforms on affine Grassmann manifolds,  Adv. Math.   \textbf{201}  (2006),   516--548.


\bibitem   {GH1} P. Goodey and R. Howard, \textit{Processes
of flats induced by higher-dimensional processes}, Adv. in Math.
\textbf{80} (1) (1990), 92--109.


\bibitem  {Gri} E. Grinberg, \textit{Radon transforms on
higher rank  Grassmannians}, J. Differential Geometry \textbf{24}
(1986), 53--68.


\bibitem  {GR}  E. Grinberg  and  B. Rubin,  \textit{Radon inversion on Grassmannians
 via G{\aa}rding-Gindikin fractional integrals}, Annals of Math.
 \textbf{159} (2004), 809--843.

\bibitem   {He}  S. Helgason, Integral geometry and Radon transform,
Springer, New York-Dordrecht-Heidelberg-London, 2011.

  \bibitem {Herz} C. Herz, \textit{Bessel functions of matrix
argument}, Ann. of Math. \textbf{61} (1955), 474--523.

\bibitem {Ig}
J. Igusa, \textit{An introduction to the theory of local zeta
functions}, AMS/IP Studies in Advanced Mathematics, \textbf{14}.
AMS, Providence, RI; International Press, Cambridge, MA, 2000.


\bibitem  {Ka} T. Kakehi, \textit{Integral geometry on Grassmann
manifolds and calculus of invariant differential operators}, J.
Funct. Anal. \textbf{168} (1999), 1-45.


\bibitem  {Kh}  S.P. Khekalo, \textit{Riesz potentials in the
space of rectangular matrices and iso-Huygens deformations of the
Cayley-Laplace operator}, Doklady Mathematics \textbf{63} (2001),
No. 1, 35--37.

\bibitem {Mat} G. Matheron, \textit{Un th\'eor\`eme d'unicit\'e pour les
 hyperplans poissoniens}, J. Appl. Probability \textbf {11}
 (1974), 184--189.



\bibitem  {Mu} R.J. Muirhead, Aspects of multivariate
statistical theory, John Wiley \& Sons. Inc., New York, 1982.

\bibitem  {OP} G. \'Olafsson and A. Pasquale, Intertwining operators and the $\cos^\lam$-transforms, Preprint, 2010.

\bibitem  {OR04}  E. Ournycheva  and B. Rubin,  \textit{Radon transform of functions  of
matrix argument}, Preprint, 2004, math.FA/0406573.

\bibitem  {OR4} \bysame, \textit{Composite cosine transforms}, Mathematika \textbf{52} (2005),
53--68.


\bibitem  {OR3}   \bysame,    \textit{The Composite Cosine Transform on
 the Stiefel Manifold and Generalized Zeta Integrals}, Contemp. Math. \textbf{405} (2006),  111--133.





\bibitem  {Pa} A. Pasquale, \textit{Maximal degenerate representations of $SL(n + 1;H)$}, J. Lie Theory \textbf{9} (1999),
 369--382.


\bibitem  {P67} E.E. Petrov, \textit{The Radon transform in
spaces of matrices and in Grassmann manifolds}, Dokl. Akad. Nauk
SSSR, \textbf{177}, No. 4 (1967), 1504--1507.



\bibitem {Ru02} B. Rubin, \textit{Inversion formulas for the spherical Radon transform and the
generalized cosine transform}, Advances in Applied Math.
\textbf{29} (2002), 471--497.

\bibitem  {Ru03} \bysame, \textit{Notes on Radon transforms in integral geometry}, Fract. Calc. Appl. Anal. \textbf{6} (2003), no. 1, 25—-72.


\bibitem  {Ru04} \bysame,  \textit{Radon transforms on affine Grassmannians}, Trans. Amer. Math. Soc. \textbf{356} (2004), 5045—-5070.

\bibitem  {Ru06} \bysame, \textit{Riesz potentials and integral geometry in the space of rectangular
matrices}, Advances in Math. \textbf{205} (2006), 549--598.


\bibitem  {Ru10} \bysame, \textit{Funk, Cosine, and Sine Transforms on Stiefel and Grassmann manifolds}, Preprint 2010, arXiv:1009.2184v1.

\bibitem {Sha} B. V. Shabat, \textit{Introduction to Complex Analysis:
Functions of Several Variables: Part II}, Translations of
Mathematical Monographs, Vol. 110,  American Mathematical Society,
1992.

\bibitem  {Sh} L.P. Shibasov,   \textit{ Integral problems in a matrix space that are
connected with the functional $ X\sp{\lambda }\sb{n,m}$.}
 Izv. Vys\v s. U\v cebn. Zaved. Matematika (1973), No.
8 (135), 101--112 (Russian).

\bibitem  {Shi} M.A. Shifrin,  Non-local inversion formula for the problem of integraal geometry for a pair of real Grassmannian manifolds (affine variant), Preprint I.P.M, Moscow, no. 91 (1981).
    
\bibitem  {Shin} T. Shintani, \textit{ On zeta-functions associated with the vector
space of quadratic forms}, J. Fac. Sci. Univ. Tokyo Sect. I A Math.
\textbf{22} (1975), 25--65.
    


   \bibitem  {T} A. Terras, Harmonic analysis on symmetric spaces
and applications, Vol. II,  Springer, Berlin, 1988.

\bibitem  {Zh1} G. Zhang,   \textit{Radon transform on real,
complex and quaternionic Grassmanians}, Duke Math. J. \textbf{138} (2007), 137--160.

\bibitem  {Zh2} \bysame, \textit{Radon, cosine and sine transforms on Grassmannian manifolds}, International Mathematics Research Notices, 2009, 1743--1772.



\end{thebibliography}

\end{document}